\documentclass[11pt,a4paper]{amsart}% use larger type; default would be 10pt

\usepackage[utf8]{inputenc}
\usepackage[english]{babel}
\usepackage{fontenc}
\usepackage{amsfonts,amsmath,amssymb,amsthm}
\usepackage{enumerate,mathtools,nicefrac,braket}
\usepackage{blindtext}
\usepackage{hyperref}

\usepackage{array}
\usepackage{tikz-cd}
\usepackage{mathrsfs}
\usepackage{verbatim}
\usepackage{float}
\usepackage{array}           % for \newcolumntype macro
\newcolumntype{C}{>{$}c<{$}} % math-mode version of "c" column type

\usepackage{csquotes}
\newcommand{\cd}{\mathrm{cd}}

\newcommand{\mb}{\mathbb}

\newcommand{\Z}{\mb Z}

\renewcommand{\phi}{\varphi}
\renewcommand{\epsilon}{\varepsilon}
\renewcommand{\bar}{\overline}

\newcommand\normal{\trianglelefteq}

\DeclareMathOperator{\lk}{lk}
\DeclareMathOperator{\st}{st}

\newtheorem{theorem}{Theorem}[section]
\newtheorem{definition}[theorem]{Definition}

\newtheorem{proposition}[theorem]{Proposition}
\newtheorem{corollary}[theorem]{Corollary}
\newtheorem{lemma}[theorem]{Lemma}

\newtheorem{example}[theorem]{Example}

\newtheorem{notation}[theorem]{Notation}
\newtheorem{conjecture}[theorem]{Conjecture}
\newtheorem{claim}[theorem]{Claim}

\newtheorem*{theorem*}{Theorem}
\newtheorem*{corollary*}{Corollary}

\begin{document}

\title[Around subgroups of Artin groups]{Around subgroups of Artin groups: derived subgroups and acylindrical hyperbolicity in the even FC-case.}

\author[J. Lopez de Gamiz Zearra]{Jone Lopez de Gamiz Zearra}
\address{Facultad de Economía y Empresa, Bilbao-Sarriko\\
Universidad del Pa\'is Vasco EHU/UPV\\
48015 Bilbao, Spain}
\email{jone.lopezdegamiz@ehu.eus}

\author[C.~Mart\'inez-P\'erez]{Conchita Mart\'inez-P\'erez}
\address{Departamento de Matem\'ticas\\
        Universidad de Zaragoza\\
         50009 Zaragoza, Spain}
\email{conmar@unizar.es}

\begin{abstract} We generalize to (certain) Artin groups some results previously known for right-angled Artin groups (RAAGs). First, we generalize a result  by  Droms, B. Servatius, and H. Servatius, and prove that the  derived subgroup of an Artin group is free if and only if the group is coherent. Second, coherent Artin groups over non complete graphs split as free amalgamated products along free abelian subgroups, and we extend to arbitrary Artin groups admitting such a splitting a recent result by Casals-Ruiz and the first author on finitely generated normal subgroups of RAAGs. Finally, we use splittings of even Artin groups of FC-type to generalize results of Minasyan and Osin on acylindrical hyperbolicity of their subgroups.
\end{abstract}

\maketitle
\section{Introduction}

Artin or Artin-Tits groups are a fascinating and rather mysterious class of groups that were first introduced by Tits \cite{Tits} as an extension of braid and Coxeter groups. They began to gain fame thanks to the works of several mathematicians such as Brieskorn, Saito and Deligne (see, for instance, \cite{BrieskornSaito} and \cite{Deligne}), and there are some specific classes (spherical-type, FC-type, right-angled Artin groups) where significant progress has been made. However, there are still many interesting open questions of  algebraic or geometric nature about the general class of Artin groups. These open problems include basic features; for instance, whether they have solvable word problem, non-trivial centers or contain torsion elements.

Artin groups can be defined in terms of labeled finite simplicial graphs, i.e.  finite labeled  graphs with no double edges nor loops and where the labels are  positive integers greater than or equal to $2$. More explicitely,
given a finite labeled simplicial graph $\Gamma$, the Artin group corresponding to $\Gamma$ is defined as 
\[A_\Gamma= \langle v\in V \mid vu\buildrel{m_e}\over{\ldots}=uv\buildrel{m_e}\over\ldots,\quad e=\{u,v\}\in E \rangle,\]
were $V$ and $E$, respectively, denote the set of vertices and edges of $\Gamma$,  $m_e$ is the label of the edge $e$ and $vu\buildrel{m_e}\over\ldots$ is the alternating word of length $m_e$ on $v$ and $u$.

Right-angled Artin groups (RAAGs for short) are one of the most studied families of Artin groups. They were first introduced by Baudisch \cite{Baudisch} in the 1970's and further developed by Droms \cite{Droms}, \cite{Droms2}, \cite{Droms3}. They play an important role in geometric group theory for many reasons. Among them, because subgroups of RAAGs provide examples of groups with interesting finiteness properties, and also because thanks to the work of Agol and Wise we know that many relevant families of groups such as hyperbolic 3-manifold groups, limit groups or torsion one-relator groups are (virtually) subgroups of RAAGs. Moreover, RAAGs are also important because there are many characterizations available of their algebraic properties in terms of combinatorial properties of the graph. An example of this type of property is Droms' following result: a RAAG $A_\Gamma$ is coherent if and only if the graph $\Gamma$ is chordal (see Section \ref{sec:coherent}). This characterization was later used by Droms together with B. Servatius and H. Servatius \cite{DromsServatiusServatius} (see Theorem \ref{RAAGscoherence_commutator}) to show that coherent RAAGs are precisely those having a free derived subgroup. Recently, Gordon \cite{Gordon} and Wise \cite{Wise} have obtained a characterization of coherent Artin groups in terms of the graph (see Theorem \ref{Artincoherence_classification}). In our first main result, we extend to Artin groups the result by Droms, B. Servatius and H. Servatius and we show 

\begin{theorem*}[Theorem \ref{Artincoherence_commutator}]
Let $\Gamma$ be a labeled finite simplicial graph. The Artin group associated to $\Gamma$, $A_\Gamma$, is coherent if and only if its derived subgroup is free.      
\end{theorem*}

Before stating some other results proved in the article, we would like to point out that in the paper, by a \emph{subgraph} we will mean \emph{induced subgraph} (see Definition \ref{def:induced}).

One of the consequences of the characterization by Gordon and Wise is that coherent Artin groups whose defining graph is not complete can be decomposed as a free amalgamated product along a free abelian subgroup. We consider more general Artin groups admitting such a decomposition in terms of proper subgraphs and we get the following.

\begin{theorem*}[Theorem \ref{Artingroup_normal}]
Let $A_\Gamma$ be an Artin group that admits a decomposition
\[A_\Gamma=A_{\Gamma_1}\ast_{A_\Delta}A_{\Gamma_2},\]
where $\Gamma_1, \Gamma_2, \Delta$ are proper subgraphs, $\Delta= \Gamma_1 \cap \Gamma_2$, $A_\Delta$ is a free abelian group and let $N\trianglelefteq A_\Gamma$ be a non-trivial finitely generated normal subgroup. Then either
\begin{itemize}
    \item $A_\Gamma \slash N$ is virtually abelian; or
    \item $A_\Gamma$ splits as a direct product of standard parabolic subgroups
  $$A_\Gamma=A_S\times A_{\Gamma \setminus S},$$
where $S$ is a subgraph of $\Gamma$, $A_S$ is a free abelian group and $N\trianglelefteq A_S$. 
\end{itemize}    
\end{theorem*}

This is a partial generalization of one of the main results in \cite{CasalsLopezdeGamiz} where the case for RAAGs is considered. One of the key properties that Casals-Ruiz and the first named author use to show that result is the existence of WPD elements and the weak malnormality of certain subgroups of RAAGs. 

In the last section, we turn to a more geometric perspective and consider the notion of  acylindrical hyperbolicity (see Section \ref{sec:acylindrical}). There are many results on the literature about acylindrical hyperbolicity of various Artin groups, and a summary can be found in \cite{CharneyMartinMorris}. In \cite{Haettel}, based on questions posted in \cite{CharneyMorris}, Haettel conjectures that if $A_\Gamma$ is an Artin group that does not decompose as a direct product of standard parabolic subgroups, then the quotient $A_\Gamma \slash Z(A_\Gamma)$ is acylindrically hyperbolic. Moreover, they  prove this conjecture for even Artin groups of FC-type. Here, we consider acylindrical hyperbolicity for subgroups of even Artin groups of FC-type.

Even Artin groups are Artin groups for which all the labels in the defining graph are even numbers. They obviously include RAAGs and share some of their nice properties. For example, as for RAAGs, subgroups associated to subgraphs (called standard parabolic subgroups) are retracts (see Subsetcion \ref{subsec:Artingroups}).
 
Even Artin groups of FC-type can be defined as even Artin groups with the condition that in any triangle in the defining graph there are always at least two edges labeled by $2$. They are even closer to RAAGs; for instance, Antol\'in and Foniqi \cite{AntolinFoniqi} have shown that, as in the class of RAAGs, an arbitrary intersection of parabolic subgroups in an even Artin group of FC-type is again parabolic, and the same property has been conjectured for arbitrary Artin groups (see \cite{CharneyMartinMorris}). Antol\'in and Foniqi's  result is essential for our next main theorem, which generalizes an analogous theorem for RAAGs by Minasyan and Osin \cite{MinasyanOsin}.

\begin{theorem*}[Theorem \ref{hyperbolic}]
    Let $G=A_\Gamma$ be an even Artin group of FC-type and $H\leq A_\Gamma$ a subgroup. Suppose that there is a vertex $w$ in $\Gamma$ such that $H$ is not a subgroup of $A_{\st_\Gamma(w)}^g$ or of $A_{\Gamma\setminus\{w\}}^g$ for any $g\in G$. Consider the decomposition of $A_\Gamma$ of the form
    \[ A_\Gamma=A_{\Gamma\setminus\{w\}}\ast_{A_{\lk_\Gamma(w)}}A_{\st_\Gamma(w)},\]
and let $T$ be the associated Bass-Serre tree. Then there is an element $h\in H$ acting hyperbolically on $T$.    
\end{theorem*}

Finally, we use this theorem to obtain the following corollary. Note that an acylindrically hyperbolic group cannot contain two infinite normal subgroups that intersect trivially (see \cite[Theorem 3.7]{MinasyanOsin}). Furthermore, even Artin groups of FC-type satisfy the $K(\pi,1)$-Conjecture (see Section \ref{subsec:Artingroups}), so they are torsion-free. Thus, Theorem \ref{hyperbolic} reflects that this obstruction yields a characterization of acylindrical hyperbolicity for subgroups of even Artin groups of FC-type.

\begin{corollary*}[Corollary \ref{acylindrically hyperbolic}]
 Let $A_\Gamma$ be an even Artin group of FC-type and $H\leq A_\Gamma$ a subgroup which is not virtually cyclic. Then either
\begin{itemize}
\item[(i)] there are $1\neq N_1,N_2\normal H$ with $N_1\cap N_2=1$ (so $N_1\times N_2\leq H$); or
\item[(ii)] $H$ is acylindrically hyperbolic.    
\end{itemize}    
\end{corollary*}

\thanks{\noindent
The first named author has been supported by the Spanish Government grant PID2020-117281GB-I00, partly by the European Regional Development Fund (ERDF), and
the Basque Government, grant IT1483-22. The second named author has been partially supported by the Spanish Government PID2021-126254NB-I00 and Departamento de Ciencia, Universidad y Sociedad del 
Conocimiento del Gobierno de Arag{\'o}n (grant code: E22-23R: ``{\'A}lgebra y Geometr{\'i}a'').
We thank the anonymous referees for their comments and suggestions that helped to improve the paper.}

\section{Preliminaries}
\subsection{Artin groups}\label{subsec:Artingroups}

If $\Gamma$ is a graph, we will denote by $V(\Gamma)$ or $V$ (if the graph is the obvious one) the vertex set of $\Gamma$ and by $E(\Gamma)$ (or $E$) its edge set. We will work with labeled finite simplicial graphs, where the labels of the edges are integers bigger than or equal to 2 and the label of the edge $e=\{u,v\}$ is denoted by $m_e$.

\begin{definition}
The \emph{Artin group} associated to $\Gamma$, $A_{\Gamma}$, is the group defined by the presentation:
\[ A_\Gamma= \langle v\in V \mid vu\buildrel{m_e}\over\ldots=uv\buildrel{m_e}\over\ldots,\quad e=\{u,v\}\in E \rangle.\]
\emph{Right-angled Artin groups (RAAGs)} are defined as Artin groups in which all the edges of the graph have label $2$.
\end{definition}

\begin{definition}
The \emph{Coxeter group} associated to $\Gamma$, $W_\Gamma$, is the quotient of $A_\Gamma$ by the relations $v^2=1, v\in V$. That is, it is defined by the presentation:
\[ W_\Gamma= \langle v\in V \mid v^2=1, vu\buildrel{m_e}\over\ldots=uv\buildrel{m_e}\over\ldots,\quad e=\{u,v\}\in E \rangle.\]
\end{definition}

We will abuse notation and write $T\subseteq \Gamma$ if $T$ is a subgraph of $\Gamma$.

\begin{definition}\label{def:induced}
We say that a subgraph $T\subseteq \Gamma$ is \emph{induced} if whenever two vertices of $T$ are linked in $\Gamma$, then they are also linked in $T$ with the same label.

All the subgraphs considered in this paper will be induced so we will just refer them as \emph{subgraphs}.    
\end{definition}

If $T\subseteq \Gamma$ is a subgraph (which in our context means induced), by \cite{Bourbaki}, the subgroup of $W_\Gamma$ generated by the vertices of $T$ is isomorphic to the Coxeter group $W_T$, and analogously, by \cite{vanderLek}, the subgroup of $A_\Gamma$ generated by the vertices of $T$ is isomorphic to the Artin group $A_T$ (and we will use $A_T$ to denote this subgroup). Also, if $S,T\subseteq\Gamma$, then as subgroups of $A_\Gamma$ we have that $A_S\cap A_T=A_{S\cap T}$  (see \cite{vanderLek}). 

\begin{definition}
If $T$ is a subgraph of $\Gamma$, the group $A_T$ (respectively $W_T$) is called a \emph{standard parabolic subgroup} of $A_\Gamma$ (respectively of $W_\Gamma$).

A \emph{parabolic subgroup} of $A_\Gamma$ is a subgroup of the form $g A_T g^{-1}$ where $g\in A_\Gamma$ and $A_T$ is a standard parabolic subgroup.
\end{definition}

\begin{definition}
We say that $A_\Gamma$ is \emph{spherical} if its associated Coxeter group, $W_\Gamma$, is finite.  
\end{definition}

Note that if $A_\Gamma$ is spherical, then the graph $\Gamma$ must be complete. A particular type of spherical Artin groups are the so called \emph{dihedral} Artin groups which correspond to the case when the graph $\Gamma$ consists of two vertices joined by an edge.

\begin{definition}
We say that $A_\Gamma$ is of \emph{FC-type} if $A_T$ is spherical for every complete subgraph $T$ of $\Gamma$.
\end{definition}

\begin{definition}
We say that $A_\Gamma$ is \emph{even} if all the labels in the graph $\Gamma$ are even numbers. \end{definition}

There is an easy characterization of even Artin groups of FC-type in terms of the defining graphs.

\begin{lemma}\cite[Lemma 2.24]{Blasco}
Let $A_\Gamma$ be an even Artin group. Then $A_\Gamma$ is of FC-type if and only if every triangular subgraph of $\Gamma$ has at least two edges labeled with $2$.    
\end{lemma}

As remarked in the introduction, the answers to several basic questions are not known in general for Artin groups, but even Artin groups of FC-type have many features in common with RAAGs. For instance, they are poly-free (which implies that they are also locally indicable and right orderable), residually finite, and they verify the $K(\pi, 1)$-Conjecture. To state this conjecture, which is one of the most important open problems about Artin groups, we recall that a CW-complex $X$ is said to be \emph{aspherical} if its only non-trivial homotopy group is the first homotopy group. For a group $\pi$, a CW-complex is a $K(\pi, 1)$ if it is aspherical and has first homotopy group the group $\pi$. 

Any Coxeter group can be represented as a discrete reflection group, i.e. a discrete group of linear transformations of a finite dimensional vector space $U$ with the generators acting as reflections with respect to some bilinear form $B$. If $W$ is a finite Coxeter group and $r\in W$ is a reflection, we define $H_r$ as the hyperplane consisting on the set of fixed points. This way, complexifying the action defined before one obtains a finite arrangement of complex hyperplanes $\mathbb{C}H_r$ in $\mathbb{C}^n$ such that $W$ acts freely in the complement, $Y_W= \mathbb{C}^n\setminus (\cup \mathbb{C}H_r)$. For infinite Coxeter groups, an analogous hyperplane complement $Y_W$ might be defined in $\mathbb{C}\otimes U$.

\begin{conjecture}[The $K(\pi,1)$-Conjecture]
Let $W$ be a Coxeter group and $A$ the associated Artin group. Then $Y_W \slash W$ is aspherical with fundamental group $A$, that is, $Y_W \slash W$ is a $K(A,1)$ space.    
\end{conjecture}

The conjecture is known to be true in several particular cases, including the cases of spherical Artin groups, RAAGs and, more generally, Artin groups of FC-type. In addition, an Artin group $A_\Gamma$ is called \emph{2-dimensional} if for any triangle subgraph $\Delta \subseteq \Gamma$, the Artin group $A_\Delta$ is not spherical. The spherical triangle Artin groups are precisely those whose associated graph has labels $(2,2,k), (2,3,3), (2,3,4)$ or $(2,3,5)$. The  $K(\pi,1)$-Conjecture is also true for $2$-dimensional Artin groups.

An equivalent formulation of the $K(\pi,1)$-Conjecture is that certain complex constructed by Salvetti, known as the \emph{Salvetti complex}, is a model for $K(\pi,1)$, \cite{CharneyDavis}. In the case of a 2-dimensional Artin group this complex is precisely the standard presentation complex, so from the fact that 2-dimensional Artin groups satisfy the $K(\pi,1)$-Conjecture, it follows that the presentation complex for these groups is aspherical.

 Another interesting and well-known property of even Artin groups is that standard parabolic subgroups are retracts (see, for example, \cite{AntolinFoniqi}). More explicitly, if $T \subseteq \Gamma$ is a subgraph, then there is a well-defined epimorphism $$\begin{aligned}
\pi\colon A_\Gamma &\to A_T\\
v&\mapsto v\text{ for }v\in V(T)\\
u&\mapsto 1\text{ for }u\not\in V(T)\\
\end{aligned}$$
which is a retract of the inclusion monomorphism $\iota\colon A_T \to A_\Gamma$.

\subsection{Coherent right-angled Artin groups}\label{sec:coherent}

Recall that a group is called \emph{coherent} if each of its finitely generated subgroups is finitely presented.

For the class of RAAGs, Droms provided a characterization of coherent RAAGs in terms of the associated graph.

\begin{theorem}\cite[Theorem 1]{Droms}\label{RAAGscoherence_classification}
If $A_\Gamma$ is a RAAG, then it is coherent if and only if $\Gamma$ is \emph{chordal}, that is, if $\Gamma$ does not contain a subgraph isomorphic to a cycle of length greater than $3$.
\end{theorem}

The proof of the previous result is based on a graph-theoretic observation that we will also use.

\begin{lemma}\label{chordalgraph}
If $\Gamma$ is chordal, then either $\Gamma$ is a complete graph or there are two subgraphs $\Gamma_1, \Gamma_2$ such that $\Gamma=\Gamma_1\cup\Gamma_2$ and $\Gamma_1 \cap \Gamma_2$ is complete. 
\end{lemma}

Another characterization of coherent RAAGs was shown later on by Droms, B. Servatius and H. Servatius which uses the derived subgroup.

\begin{theorem}\cite[Theorem 2]{DromsServatiusServatius}\label{RAAGscoherence_commutator}
If $A_\Gamma$ is a RAAG, then it is coherent if and only if its derived subgroup $[A_\Gamma, A_\Gamma]$ is free.
\end{theorem}

\section{Coherent Artin groups}
The goal of this section is to gather the multiple results that can be found in the literature to unify them and to give an analogous classification for Artin groups as the one in Theorem \ref{RAAGscoherence_classification} for RAAGs, as well as to show that Theorem \ref{RAAGscoherence_commutator} still holds in the general class of Artin groups.

\begin{figure}[ht]
\begin{tikzpicture}
\tikzstyle{pto} = [circle, minimum width=4pt, fill, inner sep=0pt]
\node[pto] (n1) at (0,1) {};
\node[pto] (n2) at (3,1) {};
\node[pto] (n3) at (1.5,0) {};
\node[pto] (n4) at (1.5,2) {};
\node (n5) at (1.3,1) {$m$};
\draw (n1) node[above] {$a$} ;
\draw (n2) node[above] {$b$} ;
\draw (n4) node[above] {$v$} ;
\draw (n3) node[below] {$w$} ;
\draw (0.70,0.55) node[below] {$2$} ;
\draw (2.30,0.55) node[below] {$2$} ;
\draw (0.7,1.45) node[above] {$2$} ;
\draw (2.3,1.45) node[above] {$2$} ;
\draw (6,1) node {$m>2$} ;

%\draw (n1) node[above right] {$v_2$} ;
%\draw (n2) node[below left] {$v_0$} ;
%\draw (n3) node[below right] {$v_3$} ;
%\draw (n4) node[above left] {$v_1$} ;
\draw (n4)--(n1) -- (n3) -- (n2) -- (n4) -- (n3);
\end{tikzpicture}
\caption{}
\label{fig:conditioniii}
\end{figure}

Recall that Droms characterized coherent RAAGs in \cite{Droms}. Subsequently, Gordon \cite{Gordon} gave a classification of coherent Artin groups that depended upon the incoherence of the Artin group with associated labeled graph being a triangle with labels $(2,3,5)$. Finally, Wise proved in \cite{Wise} that such an Artin group is not coherent, maintaining the simplicity of Gordon's classification.

\begin{theorem}\cite{Droms,Gordon,Wise}
\label{Artincoherence_classification} Let $\Gamma$ be a labeled finite simplicial graph. The Artin group associated to $\Gamma$, $A_\Gamma$, is coherent if and only if the following three conditions hold:
\begin{itemize}
\item[(i)] $\Gamma$ is chordal;
\item[(ii)] for any complete subgraph $\Delta\subseteq\Gamma$ of 3 or 4 vertices, $\Delta$ has at most one edge with a label different to $2$;
\item[(iii)] $\Gamma$ has no subgraph as in Figure \ref{fig:conditioniii}.
\end{itemize}
\end{theorem}

Using this characterization of coherent Artin groups, we now show a result that, although is not explicitly stated in \cite{Gordon}, it is implicitly used and will be helpful throughout the article.

\begin{proposition}\label{Artincoherent_decomposition}
Let $A_\Gamma$ be a coherent Artin group. Then either
\begin{itemize}
    \item $\Gamma$ is a complete graph and $A_\Gamma$ is the direct product of a dihedral Artin group and a free abelian group; or
    \item $A_\Gamma$ decomposes as an amalgamated free product \[A_\Gamma= A_{\Gamma_1} \ast_{A_\Delta} A_{\Gamma_2},\]
    where $\Gamma_1, \Gamma_2, \Delta$ are subgraphs, $\Delta= \Gamma_1 \cap \Gamma_2$ and $A_\Delta$ is a free abelian group.
\end{itemize}
\end{proposition}

\begin{proof}
Let $\Gamma$ be a labeled finite simplicial graph with conditions (i), (ii) and (iii) from the statement of Theorem \ref{Artincoherence_classification}. Firstly, if $\Gamma$ is complete, then (ii) implies that $\Gamma$ has at most one edge $e$ such that $m_e\neq 2$. Hence, $A_\Gamma$ is the direct product of at most one dihedral Artin group and a free abelian group of finite rank.
Secondly, if $\Gamma$ is not complete, from condition (i) and Lemma \ref{chordalgraph} we get that there are two subgraphs $\Gamma_1, \Gamma_2$ such that $\Gamma= \Gamma_1 \cup \Gamma_2$ and $\Delta \coloneqq \Gamma_1 \cap \Gamma_2$ is complete. The Artin group $A_\Delta$ is also coherent, so $\Delta$ has at most one edge $e=\{w,v\}$ such that $m_e \neq 2$.

If all the labels of $\Delta$ are $2$, then we get the statement. Otherwise, we claim that we may find a decomposition of $\Gamma$, $\Gamma=\hat{\Gamma}_1 \cup \hat{\Gamma}_2$, such that $\hat{\Gamma}_1$ and $\hat{\Gamma}_2$ are subgraphs and $\hat{\Gamma}_1 \cap \hat{\Gamma}_2 \subseteq \Delta$ is a complete subgraph with all the labels equal to $2$.

Indeed, assume first that $v$ is not linked to any vertex in $\Gamma_2 \setminus \Delta$. Then, taking $\hat{\Gamma}_2$ to be $\Gamma_2 \setminus v$ and $\hat{\Gamma}_1$ to be $\Gamma_1$, we do obtain the claim since $\hat{\Gamma}_1 \cap \hat{\Gamma}_2= \Delta \setminus v$ is now a complete subgraph with all the edges labeled by $2$. As a consequence, we may assume that there is a vertex $u$ in $\Gamma_2 \setminus \Delta$ adjacent to $v$ and let us denote by $\Gamma_u$ the connected component of $\Gamma_2 \setminus \Delta$ containing $u$.

We again distinguish two subcases here. Suppose that $w$ is not joined to any vertex of $\Gamma_u$. Then, we might take $\hat{\Gamma}_1$ to be $\Gamma_1 \cup (\Gamma_2 \setminus \Gamma_u)$ and $\hat{\Gamma}_2$ to be $\Gamma_u \cup (\Delta \setminus w)$ so that $\hat{\Gamma}_1 \cap \hat{\Gamma}_2= \Delta \setminus w$ is a complete subgraph with all the edges labeled by $2$.

If there is a vertex $t$ in $\Gamma_u$ adjacent to $w$, then there is a cycle through $v, w, t$ and $u$. We claim that there is some vertex $a$ in $\Gamma_2 \setminus \Delta$ linked to both $v$ and $w$. Condition (i) implies that there is a chord in the cycle through $v, w, t$ and $u$. That chord yields a shorter cycle through $v, w$ and $\Gamma_u$, so repeating this argument periodically, we eventually get the claim.  Finally,  condition (ii) forces the labels between $a$ and $v$ and between $a$ and $w$ to be equal to $2$.

We might repeat the same argument as before replacing $\Gamma_2$ by $\Gamma_1$ and deduce that there is a vertex $b$ in $\Gamma_1 \setminus \Delta$ linked to both $v$ and $w$ with labels equal to $2$. But then the vertices $v, w, a, b$ form a subgraph as in Figure \ref{fig:conditioniii}, which is a contradiction.
\end{proof}

We end up this section by generalizing Theorem \ref{RAAGscoherence_commutator} to the general class of Artin groups, namely:

\begin{theorem}\label{Artincoherence_commutator}
Let $\Gamma$ be a labeled finite simplicial graph. The Artin group associated to $\Gamma$, $A_\Gamma$, is coherent if and only if its derived subgroup is free.    
\end{theorem}

We first prove that if an Artin group is coherent, then its derived subgroup is free.

For arbitrary RAAGs, it is well-known that the abelianization of a standard parabolic subgroup coincides with its intersection with the abelianization of the ambient group (see \cite[Section 4]{DromsServatiusServatius}). We first extend this fact to the class of even Artin groups.

\begin{notation}
If $\Gamma$ is a labeled finite simplicial graph, let us denote by $\Gamma^{odd}$ the graph obtained from $\Gamma$ by deleting the edges with even label.    
\end{notation}

\begin{proposition}\label{EvenArtin_commutator}
Let $A_\Gamma$ be an even Artin group and let $\Omega \subseteq \Gamma$ be a subgraph. Then $[A_\Omega, A_\Omega]= A_\Omega \cap [A_\Gamma, A_\Gamma]$.    
\end{proposition}

\begin{proof}
The inclusion $[A_\Omega, A_\Omega]\subseteq  A_\Omega \cap [A_\Gamma, A_\Gamma]$ is obvious, so let us focus on the other inclusion.

Let $\Delta$ be the complete labeled simplicial graph with the same number of vertices as $\Gamma$ and such that all the labels are $2$. As we are assuming that $A_\Gamma$ is even, there is a well-defined epimorphism $\varphi\colon A_\Gamma\to A_\Delta$ which is precisely the abelianization map. Therefore $[A_\Gamma, A_\Gamma]= \ker \varphi$. But using again the fact that $A_\Gamma$ is even we deduce that the restriction of $\phi$ to $A_\Omega$ is also the abelianization map of $A_\Omega$, so
\[ [A_\Omega, A_\Omega]= \ker \varphi_{|A_\Omega} = \ker \varphi \cap A_\Omega= [A_\Gamma, A_\Gamma] \cap A_\Omega.\]
\end{proof}

We furthermore show that if $A_\Gamma$ is a coherent Artin group with the decomposition as in Proposition \ref{Artincoherent_decomposition}, then the conclusion of Proposition \ref{EvenArtin_commutator} still remains true for the standard parabolic subgroups $A_{\Gamma_1}, A_{\Gamma_2}$ and $A_{\Delta}$. But before that, it will be useful to state the following (probably well-known) result that describes the abelianization of an arbitrary Artin group.

\begin{lemma}\label{lem:abelianization} Let $A_\Gamma$ be an Artin group and let $Z$ be the complete labeled simplicial graph with the same number of vertices as connected components in $\Gamma^{odd}$ and such that all the labels are $2$. Then 
$$A_\Gamma \slash [A_\Gamma,A_\Gamma]\cong A_Z.$$
\end{lemma}
\begin{proof} For each vertex $v$ of $\Gamma$, let $\varphi(v)$ be the vertex in $Z$ corresponding to the connected component of $\Gamma^{odd}$ that contains $v$. Then there is a well-defined group homomorphism
$$\begin{aligned}\varphi\colon A_\Gamma&\to A_Z\\
v&\mapsto\varphi(v).\\
\end{aligned}$$
Indeed, note that every defining relation in $A_\Gamma$ corresponds to an edge in $\Gamma$ between two vertices, say $v$ and $w$. If $\varphi(v)=\varphi(w)$, then the relation is obviously preserved, and this is always the case if the label is odd. If $\varphi(v)\neq\varphi(w)$, then the label must be even and, by the definition of $Z$, $\varphi(v)\varphi(w)=\varphi(w)\varphi(v)$. Hence, the relation is also preserved in this case. Moreover, $\varphi$ is obviously surjective. Given any map $\psi\colon A_\Gamma\to A$ where $A$ is an abelian group, if $v$ and $w$ are vertices in $\Gamma$ linked by an edge with odd label, the fact that $A$ is abelian implies that $\psi(v)=\psi(w)$. Hence, the same holds true for all the vertices in the same connected component of  $\Gamma^{odd}$, and this means that $\psi$ factors through $\varphi$.
\end{proof}

\begin{lemma}\label{Artin_parabolic_commutator}
Let $A_\Gamma$ be a coherent Artin group such that $\Gamma$ is not complete and let us consider the decomposition
\[ A_\Gamma= A_{\Gamma_1} \ast_{A_\Delta} A_{\Gamma_2}\]
as in Proposition \ref{Artincoherent_decomposition}. Then $[A_\Omega, A_\Omega]= A_\Omega \cap [A_\Gamma, A_\Gamma]$ where $\Omega$ is either $\Gamma_1, \Gamma_2$ or $\Delta$.
\end{lemma}

\begin{proof}
The inclusion $[A_\Omega, A_\Omega]\subseteq  A_\Omega \cap [A_\Gamma, A_\Gamma]$ is obvious. Consider the map $\varphi$ defined in Lemma \ref{lem:abelianization}. The proof of that result implies that $\varphi$ is the abelianization map and that $[A_\Gamma, A_\Gamma]= \ker \varphi$. It remains to show that $\varphi_{|A_\Omega}$ is the abelianization map of the standard parabolic subgroup $A_\Omega$. For that aim, it suffices to show that the connected components of $\Omega^{odd}$ are precisely the intersections of the connected components of $\Gamma^{odd}$ with $\Omega$.

Assume, by contradiction, that this is not the case, so that there are vertices $v$ and $w$ in $\Omega$ lying in different connected components of $\Omega^{odd}$ but such that there is a path $p$ in $\Gamma$ from $v$ to $w$ where the labels of the edges in the path are odd. Without loss of generality, we further assume that $p$ is the shortest path with these properties, that is, $p$ is the shortest path in $\Gamma$ with all labels odd joining two vertices of $\Omega$ in different connected components of $\Omega^{odd}$. Since we are assuming that $v$ and $w$ lie in different connected components of $\Omega^{odd}$, at least one vertex $u$ of the path $p$ does not belong to $\Omega$. Moreover, the fact that the vertices in the subgraph $\Gamma_1$ are connected to vertices in $\Gamma_2$ only through $\Delta$ implies that there must be some vertex $x$ in $p$ between $v$ and $u$ that lies in $\Delta$ and, in fact, the choice of $p$ implies that $x=v$. Similarly, there must be some vertex $y$ that belongs to $\Delta$ between $w$ and $u$. As we are assuming that $v$ and $w$ are not connected by a path with only edges with odd labels inside $\Omega$, it is not possible to have that $x$ equals $y$. Actually, again by the choice of the path $p$, we must have that $w=y$.

The graph $\Delta$ is complete, so $v$ and $w$ are adjacent and we have a cycle with all labels but one odd. As $\Gamma$ is chordal, by triangulating this cycle we would eventually get a triangle with two edges labeled by odd numbers, and this is a contradiction because $A_ \Gamma$ is coherent. 
\end{proof}

\begin{example}
The previous result does not hold for arbitrary standard parabolic subgroups, not even in the case of coherent Artin groups.

Consider, for instance, the graph $\Gamma$ with 3 vertices $v,u,w$ and two edges, $\{u,v\}$ and $\{v,w\}$ both labeled by $3$. Then, from the relation $vuv=uvu$ we obtain that \[[v,u]=uv^{-1}\in[A_\Gamma,A_\Gamma],\] and similarly, $vw^{-1}\in[A_\Gamma,A_\Gamma]$.

As a consequence, the element $uw^{-1}$ lies in the intersection $[A_\Gamma,A_\Gamma]\cap A_{\Omega}$, where $\Omega$ is the subgraph with vertices $u$ and $w$. But $A_{\Omega}$ is a free group and $uw^{-1}$ does not belong to $[A_\Omega,A_\Omega]$.
\end{example}

We will also need the following well-known result.

\begin{lemma}\cite{KarrassSolitar}\label{freeamalgamated} Let $G=G_1\ast_H G_2$ be a free amalgamated product and let $1\neq N\trianglelefteq G$ be a normal subgroup such that $N\cap H=1$ and both intersections $N\cap G_1$, $N\cap G_2$ are free or trivial. Then the subgroup $N$ is free.
\end{lemma}

At this point we have all the necessary tools to prove one of the implications of Theorem \ref{Artincoherence_commutator}. We begin with the special case of dihedral Artin groups.

\begin{example}\label{ex:dihedral} Dihedral Artin groups are well-known to be free-by-cyclic: in fact this is also a way to see that they are coherent \cite{FeighnHandel}. Being free-by-cyclic implies that the derived subgroup is free. This is also a consequence of the explicit computations performed on
\cite[Theorems 3.6, 3.9 and 3.18]{MulhollandRolfsen}.
\end{example}

\begin{theorem} Let $A_\Gamma$ be a coherent Artin group. Then the derived subgroup $[A_\Gamma, A_\Gamma]$ is free or trivial. 
\end{theorem}

\begin{proof} If $A_\Gamma$ is abelian, $[A_\Gamma, A_\Gamma]$ is trivial.
If $A_\Gamma$ is a non abelian coherent Artin group, from Proposition \ref{Artincoherent_decomposition} we know that either $A_\Gamma$ is the direct product of a dihedral Artin group and a free abelian group or it has a decomposition as in the second bullet point. In the first case, the result is trivial by Example \ref{ex:dihedral}.

In the second case, let us assume that $A_\Gamma$ decomposes as 
\[A_\Gamma=A_{\Gamma_1}\ast_{A_\Delta}A_{\Gamma_2},\]
with $\Gamma_1,\Gamma_1,\Delta$ subgraphs of $\Gamma$, $\Delta=\Gamma_1\cap\Gamma_2$ and $A_\Delta$ is a free abelian group.

Lemma \ref{Artin_parabolic_commutator} implies that $[A_\Gamma, A_\Gamma]\cap A_\Delta= [A_\Delta, A_\Delta]=1$. Moreover, by induction on the number of vertices we may suppose that $[A_{\Gamma_1},A_{\Gamma_1}]=[A_\Gamma, A_\Gamma]\cap A_{\Gamma_1}$ and $[A_{\Gamma_2},A_{\Gamma_2}]=[A_\Gamma, A_\Gamma]\cap A_{\Gamma_2}$ are either free or trivial groups. In conclusion, from Lemma \ref{freeamalgamated} we deduce that the derived subgroup $[A_\Gamma, A_\Gamma]$ is also free.    
\end{proof}

To prove the converse of this result we will use the following technical result.

\begin{lemma}\label{lem:aspherical} Let $G$ be a group with a finite aspherical presentation $G=\langle X\mid R\rangle$ such that $|X|\leq|R|$ and let $N\trianglelefteq G$ be a normal subgroup where $G\slash N$ is a torsion-free abelian group. Then $N$ is not free and is not finitely presented.
\end{lemma}
\begin{proof}
Let us consider the chain complex of the universal cover of the presentation complex which has an exact sequence of the form
$$\oplus_{R}\Z G\to \oplus_{X}\Z G\to \Z G.$$
Tensoring this sequence with $\Z [G\slash N]$  we obtain
$$\oplus_{R}\Z [G\slash N]\to \oplus_{X}\Z [G\slash N]\to \Z [G\slash N],$$
a sequence where the homology groups are precisely the homology groups of $N$. The ring $\Z [G\slash N]$ has no zero divisors  (the zero divisor conjecture holds for torsion-free polycyclic groups, see \cite{FarkasSnider}) so it embeds in its field of fractions, say $F$. The field $F$ is a flat $\Z [G\slash N]$-module, so tensoring with $F$ preserves exact sequences, and the homology of the sequence
$$\oplus_{R}F\buildrel{\delta_2}\over\to \oplus_{X}F\buildrel{\delta_1}\over\to F$$
is precisely $H_\bullet(N)\otimes_{G\slash N}F$.

Note that $H_0(N)=\Z$ is a torsion $\Z [G\slash N]$-module (the elements of the form $1-\bar g$ annihilate it) so it vanishes when we tensor it with the field of fractions, i.e. $H_0(N)\otimes_{G\slash N}F=0$. Hence, the map $\delta_1$ is an epimorphism. As this is a sequence of $F$-vector spaces and $\mathrm{im}\delta_2\leq\ker\delta_1$, we must have that
$$|R|-\dim\ker\delta_2=\dim\mathrm{im}\delta_2\leq\dim\ker\delta_1=|X|-1,$$ 
and using that $|X|\leq|R|$, we deduce that
$$1\leq|R|-|X|+1\leq \dim \ker\delta_2.$$
Therefore, $H_2(N)\otimes_{G\slash N}F\neq 0$, which implies that $H_2(N)\neq 0$. Thus, the group $N$ is not free and, moreover, $H_2(N)$ cannot be a finitely generated abelian group because, in that case, it would be a torsion $\Z[G\slash N]$-module.  This implies that $N$ is not finitely presented. Otherwise, one could construct a free resolution with finitely generated modules up to degree 2, which would force $H_2(N)$ to be finitely generated.
\end{proof}

\begin{proposition} Let $A_\Gamma$ be an incoherent Artin group. Then the derived subgroup $[A_\Gamma,A_\Gamma]$ is not free. 
\end{proposition}
\begin{proof} By Theorem \ref{Artincoherence_classification}, we may assume that $\Gamma$ has a subgraph $\Delta$ of the following form:
\begin{itemize}
\item[(i)] the graph in Figure \ref{fig:conditioniii},

\item[(ii)] a cycle of length at least $4$,

\item[(iii)] a triangle with at most one label $2$. 
\end{itemize}
Moreover, as $[A_\Delta,A_\Delta]$ is a subgroup of $[A_\Gamma,A_\Gamma]$, it suffices to show that $[A_\Delta,A_\Delta]$  is not free. In summary, we may assume that $\Gamma$ is either of the form (i), (ii) or (iii).

If $\Gamma$ is the graph of Figure \ref{fig:conditioniii}, then the elements $[a,b]$ and $[v,w]$ lie in $[A_\Gamma,A_\Gamma]$ and they commute, so $[A_\Gamma,A_\Gamma]$ is not free.

If $\Gamma$ is as in (iii), we distinguish two cases depending on whether $A_\Gamma$ is spherical or not. The spherical cases are precisely when the labels of the triangle are $(2,3,3)$, $(2,3,4)$ or $(2,3,5)$ (see, for instance, \cite{Kasia}). In the $(2,3,4)$ case, from \cite[Corollary 1.5]{EscartinMartinez} we see that the derived subgroup is finitely generated but not finitely presented, so it is not free. In the other two remaining cases, both groups have infinite cyclic abelianizations. If the derived subgroup was free, we would have that
$$\cd(A_\Gamma)\leq\cd(A_\Gamma \slash [A_\Gamma,A_\Gamma])+\cd([A_\Gamma,A_\Gamma])=2.$$
The cohomological dimension of spherical type Artin groups is the number of standard generators (see \cite{BrieskornSaito}, \cite{Deligne}), so this leads to a contradiction.

This means that $A_\Gamma$ is $2$-dimensional (for non spherical triangle Artin groups see \cite[Lemma 2.2]{Gordon}), so $A_\Gamma$ satisfies the $K(\pi,1)$-Conjecture. Hence, the presentation complex of the standard presentation is aspherical. The fact that $\Gamma$ is a cycle implies that if $n$ is the number of vertices, then this presentation has precisely $n$ generators and $n$ relations, so from Lemma \ref{lem:aspherical} we deduce that the derived subgroup is not free (observe that Lemma \ref{lem:abelianization} implies that the abelianization of $A_\Gamma$ is torsion-free). 
\end{proof}

\section{Finitely generated normal subgroups of coherent Artin groups}

Recall from the introduction that one of the goals of the paper is to show that finitely generated normal subgroups of certain Artin groups behave as in the class of RAAGs. In \cite{CasalsLopezdeGamiz}, it was proved that the quotient of a RAAG by a finitely generated (full) normal subgroup is virtually abelian, that is, that if $G$ is a RAAG and $N\trianglelefteq G$ is a finitely generated (full) normal subgroup of $G$, then $G \slash N$ is virtually abelian.

In this section we prove the analogous version for those Artin groups that split as a free amalgamated product over a free abelian standard parabolic subgroup. Recall that by Proposition \ref{Artincoherent_decomposition}, coherent Artin groups admit such a decomposition. Namely, we show

\begin{theorem}\label{Artingroup_normal}
Let $A_\Gamma$ be an Artin group that admits a decomposition
\[A_\Gamma=A_{\Gamma_1}\ast_{A_\Delta}A_{\Gamma_2},\]
where $\Gamma_1, \Gamma_2, \Delta$ are proper subgraphs, $\Delta= \Gamma_1 \cap \Gamma_2$, $A_\Delta$ is a free abelian group and let $N\trianglelefteq A_\Gamma$ be a non-trivial finitely generated normal subgroup. Then either
\begin{itemize}
    \item $A_\Gamma \slash N$ is virtually abelian; or
    \item $A_\Gamma$ splits as a direct product of standard parabolic subgroups
  $$A_\Gamma=A_S\times A_{\Gamma \setminus S},$$
where $S$ is a subgraph of $\Gamma$, $A_S$ is a free abelian group and $N\trianglelefteq A_S$. 
\end{itemize}
\end{theorem}

In order to show this result, we first need to understand the kernel of the action of $A_\Gamma$ in the Bass-Serre tree corresponding to the decomposition of the statement.

\begin{lemma}\label{Artingroup_kernel}
Let $A_\Gamma$ be an Artin group that admits a decomposition
\[A_\Gamma=A_{\Gamma_1}\ast_{A_\Delta}A_{\Gamma_2},\]
where $\Gamma_1, \Gamma_2, \Delta$ are proper subgraphs, $\Delta= \Gamma_1 \cap \Gamma_2$ and $A_\Delta$ is a free abelian group. Denote by $T$ the Bass-Serre tree associated to this decomposition. Then the kernel $K$ of the action of $A_\Gamma$ on $T$ is a standard parabolic subgroup of $A_\Gamma$ which is a free abelian direct factor. That is, $K= A_S$ where $S$ is a subgraph of $\Gamma$, $K$ is a free abelian group and \[A_\Gamma=A_S\times A_{\Gamma\setminus S}.\]
\end{lemma}

\begin{proof} For a vertex $w$ of $\Gamma$, let us define the set
$$\lk_\Gamma^2(w)=\{v\in V \mid v\text{ is adjacent to $w$ with an edge labeled by 2}\}.$$
Define the subgraph $S$ to be
$$S=\bigcap_{w\in V(\Gamma\setminus\Delta)}\lk^2_\Gamma(w)\cap\Delta.$$
We claim that for any $w\in V(\Gamma\setminus\Delta)$, 
$$A_\Delta\cap A_\Delta^{w^2}=A_{\lk^2_\Gamma(w)\cap\Delta}.$$
The reason why we use $w^2$ is technical and it will become clear later.
The claim implies  the result. Indeed, it implies that $K$, which by definition equals
\[ K= \bigcap_{g\in A_\Gamma} A_\Delta^g,\]
is a subgroup of
\[ \bigcap_{w\in V(\Gamma\setminus\Delta)}(A_\Delta\cap A_\Delta^{w^2})=\bigcap_{w\in V(\Gamma\setminus\Delta)}A_{\lk^2_\Gamma(w)\cap\Delta}=A_S,\] 
 and since $A_S$ is a subgroup of $K$, we get that $K= A_S$.

Let us now prove the claim. Note that we may actually assume that $A_\Gamma$ is generated by $w$ and $\Delta$, that is, we may assume that $\Gamma= \Delta \cup \{w\}$ and that $S$ equals $\lk^2_\Gamma(w)\cap\Delta$. In this setting, we may split the groups as follows:
\[ A_\Delta= A_S \times A_{\Delta_1} , \quad A_\Gamma= A_S \times A_T,\]
with $\Delta_1= \Delta \setminus S$ and $T= \Delta_1 \cup \{w\}.$ Hence, if we show that in the group $A_T$ we have that $A_{\Delta_1} \cap A_{\Delta_1}^{w^2}=1$, then we obtain that $A_\Delta \cap A_\Delta^{w^2}= A_S$.

Once again, this allows us to restrict ourselves to the scenario where $\Gamma$ is $T$, $S= \varnothing$ and $\Delta= \Delta_1$, that is, we may suppose that no vertex of $\Delta$ is linked to $w$ with an edge labeled by $2$ and we have to show that  $A_\Delta \cap A_\Delta^{w^2}= 1$.

Let $A$ be the subgroup generated by $w$ and $\lk_\Delta(w)$. Then $A$ is an Artin group based on a complete graph so that all edges having $w$ as a vertex have label strictly greater than $2$ and all the rest of the edges have label 2 (note that $A$ could just be the cyclic group generated by $w$). Let $C$ be the free abelian subgroup generated by $\lk_\Delta(w)$. We could have that $C=A_\Delta$, and, in fact, we want to reduce the problem to that case.

We have an amalgamated free product decomposition
$$A_\Gamma=A\ast_CA_\Delta,$$
and let $T$ be the associated Bass-Serre tree.
Suppose that $g$ is an element in $A_\Delta\cap A_\Delta^{w^2}$. Then, $g$ would equal $h^{w^2}$ for some $h\in A_\Delta$, so $g^{-1}w^{-2}hw^2=1$. This implies that in the tree $T$ there is a path with vertices
\[ A_\Delta= g^{-1}A_\Delta, \quad g^{-1}A= g^{-1}w^{-2}A, \quad g^{-1}w^{-2}A_\Delta= g^{-1}w^{-2}hA_\Delta,\]
\[ g^{-1}w^{-2}hA= g^{-1}w^{-2}hw^2A=A, \quad A_\Delta.\]
Since $T$ is a tree, there is a backtracking in this path, and taking into account that $w^2$ is not an element of $A_\Delta$, the only option is to have that $g^{-1}A=A= g^{-1}w^{-2}A.$

Therefore, $g$ and $h$ both lie in $A$, so they also belong to $A\cap A_\Delta= C.$ Hence, $g\in C \cap C^{w^2}$ and, as a consequence, $C\cap C^{w^2}=A_\Delta\cap A_\Delta^{w^2}$. In conclusion, we may assume that $A_\Gamma=A$ and $C=A_\Delta$, i.e., that all the vertices in $\Delta$ are linked to $w$ and labels are strictly bigger than 2.   

At this point, we are going to use the result conjectured by Tits and proved by Crisp-Paris in  \cite[Theorem 1]{ParisCrisp} stating that, in an Artin group, the squares of the standard generators generate a RAAG with relators being the commutators between those vertices that already commuted in the ambient group. More precisely, if we consider the subgroup $C_1$ of $C=A_\Delta$ generated by $t^2$ for $t\in\Delta$ and the subgroup $A_1$ of $A$ generated by $w^2$ and $C_1$, then by \cite[Theorem 1]{ParisCrisp}, $A_1$ is the free product of the cyclic group generated by $w^2$ and the free abelian group $C_1$. Thus $C_1$ is malnormal in $A_1$, so in particular $C_1\cap C_1^{w^2}=1$. As the subgroups $C_1$ and $C_1^{w^2}$ have finite index in $C$ and $C^{w^2}$, respectively,  $1=C_1\cap C_1^{w^2}$ has finite index in $A_\Delta\cap A_\Delta^{w^2}$, so $A_\Delta\cap A_\Delta^{w^2}=1$.

\end{proof}

We are now ready to prove the main result of the section.

\begin{proof}[Proof of Theorem \ref{Artingroup_normal}]
Let $A_\Gamma$ be as in the statement, let $T$ be the associated Bass-Serre tree and let $N$ be a non-trivial finitely generated normal subgroup of $A_\Gamma$.

The group $N$ is a subgroup of $A_\Gamma$, so it also acts on $T$ and by \cite[Theorem 2.1]{Ratcliffe} we have that either $NA_\Delta$ has finite index in $A_\Gamma$ or $N$ lies in $A_\Delta$. In the former case, since $A_\Delta$ is a free abelian group, we get that $A_\Gamma \slash N$ is virtually abelian. In the later case, if $N \trianglelefteq A_\Delta$, as $N$ is normal, it actually lies in $K$, the kernel of the action of $A_\Gamma$ on $T$, so the result follows from Lemma \ref{Artingroup_kernel}.
\end{proof}

\section{Artin groups of FC-type and parabolic subgroups}\label{sec:acylindrical}

The main result of the previous section is based on the fact that for Artin groups admiting a decomposition as a free amalgamated product where the amalgamated subgroup is standard parabolic and free abelian, one can understand the kernel of the action of the group on the associated Bass-Serre tree. The kernel is an intersection of parabolic subgroups, and as stated in the introduction, it has been conjectured that for Artin groups, any intersection of parabolic subgroups is again parabolic. This conjecture has been proved for even Artin groups of FC-type by Antolín and Foniqi in \cite[Corollary 5.1]{AntolinFoniqi} and it is known for some other families. An updated status can be found in \cite{CharneyMartinMorris}, where Charney, Martin and Morris-Wright show that in irreducible Artin groups where intersections of parabolics are also parabolics, the parabolic subgroups are weakly malnormal.

Recall that a subgroup $H$ of a group $G$ is \emph{weakly malnormal} if there is some element $g\in G$ such that $H\cap H^g=1$ and an Artin group is \emph{irreducible} if it does not split as a direct product of standard parabolic subgroups. As a consequence, Charney, Martin and Morris-Wright deduce in \cite{CharneyMartinMorris} that these two properties (that is, being irreducible and having the parabolic intersection property) also imply that the group is acylindrically hyperbolic. In particular, they obtain that irreducible even Artin groups of FC-type are acylindrically hyperbolic. Recall \cite[Definition 1.1]{MinasyanOsin} that a group $G$ is called \emph{acylindrically hyperbolic} is there is a possibly infinite generating set $X$ such that the associated Cayley graph $\mathcal{C}$ is hyperbolic with the boundary consisting of more than 2 points and, moreover, the action of $G$ on $\mathcal{C}$ is \emph{acylindrical}, meaning that for every $\epsilon>0$ there are $R,N>0$ such that whenever $x,y\in\mathcal{C}$ are such that $d(x,y)\geq R$, then there are at most $N$ elements $g\in G$ with $d(x,gx)\leq \epsilon$ and $d(y,gy)\leq\epsilon$. 

In this section we will consider the problem of acylindrical hyperbolicity for certain subgroups of even Artin groups of FC-type, generalizing results of Minasyan and Osin for the family of RAAGs (see \cite{MinasyanOsin}). 

We first need a couple of technical results about normalizers and centralizers. Given an Artin group $G=A_\Gamma$ and a standard parabolic subgroup $A_S$, $S\subseteq\Gamma$, Godelle determined in \cite[Corollary 0.4]{Godelle} a generating family for the \emph{quasi-centralizer}, $QC_G(A_S)$, which is the subgroup
$$QC_G(A_S)=\{g\in G\mid gS=Sg\}.$$
As Antolín and Foniqi have observed in \cite[Lemma 2.13]{AntolinFoniqi2}, if $A_\Gamma$ is an even Artin group of FC-type, then the quasi-centralizer coincides with the centralizer $C_G(A_S)$ and, as a consequence, they prove in \cite[Theorem 2.12, Lemma 2.13]{AntolinFoniqi2} the  equalities
\begin{equation}\label{eq:AF}N_{G}(A_S)=A_S\cdot C_{G}(A_S)=A_S\bigcap_{a\in S}C_G(a).
\end{equation}

\begin{lemma}\label{tec_Godelle} Let $G=A_\Gamma$ be an even Artin group of FC-type and $a\in V$. Then the centralizer $C_G(a)$ lies in the parabolic subgroup $A_{\st_\Gamma(a)}$ and there is a (non neccesarily parabolic) subgroup $L$ which is an Artin group based on a proper subgraph of $\Gamma$ such that
$$C_G(a)=\langle a\rangle\ltimes L.$$
Moreover, if $b\in V$ is not adjacent to $a$, then $C_G(a)\cap C_G(b)$ lies in the parabolic subgroup $A_T$, where $T$ is the subgraph induced by
$$\{v\in V \mid [v,a]=[v,b]=1\}.$$
\end{lemma}

\begin{proof}
As we noted above, as $\Gamma$ is even we can use \cite[Corollary 0.4]{Godelle} to obtain a generating system for the centralizer of any standard parabolic subgroup $A_S$. More explicitely, Godelle proves that there is a generating system consisting on \emph{elementary ribbons}, which are elements $g$ of a particular type such that $g^{-1}Xg=Y$ where $X,Y\subseteq\Gamma$. The fact that $\Gamma$ is even implies that if there is such an element $g$, then we must have that $X=Y$. To see this, assume that there is some $x\in X$, $x\not\in Y$ such that  $g^{-1}xg$ equals $y\in Y$. Applying the retraction $\pi \colon A_\Gamma\to A_Y$ as at the end of Subsection \ref{subsec:Artingroups}, we get that $y=\pi(y)=\pi(g^{-1}xg)=\pi(g^{-1})\pi(g)=1$, which is a contradiction. This shows that $X\subseteq Y$, and by symmetry, we deduce that $X=Y$. The elementary ribbons are described in terms of the group elements $\Delta_X\in G$ where $X\subseteq\Gamma$ is a spherical subset. We do not need a general description of $\Delta_X$, but just the fact that if $X=\{a\}$, then $\Delta_X=a$ and if $X=\{a,w\}$ is an edge with label $2k$, then $\Delta_X=(wa)^k$. By \cite[Corollary 0.4]{Godelle}, in even Artin groups the elementary ribbons that generate $C_G(A_S)$ are of one of the following two types:
\begin{enumerate}
\item $\Delta_X$ for $A_X$ an indecomposable spherical direct factor of $A_S$; or
\item $\Delta_X\Delta^{-1}_{X\setminus\{t\}}$ if there is $t\in V$ such that the indecomposable component $X$ of $S\cup\{t\}$ containing $t$ is spherical.
\end{enumerate}

Applying this to the case when $S=\{a\}$, we see that $C_G(a)$ is generated by $a$, all those $v\in V$ such that $[v,a]=1$ and an element of the form $(wa)^ka^{-1}$ for each $w\in V$ linked with $a$ with label $2k$ for $k>1$. 
As all these elements lie in $A_{\st_\Gamma(a)}$, so does the subgroup $C_G(a)$.

To check that $C_G(a)$ can be described as in the statement, it will be useful to modify this generating set. The subgroup $C_G(a)$ is also generated by $a$, all those $v\in V$ such that $[v,a]=1$ and an element of the form 
\begin{equation}\label{eq:defz}
z_{a,w}=a^{1-k}(wa)^ka^{-1}=w^{a^{k-1}}w^{a^{k-2}}\ldots w,
\end{equation} 
for each $w\in V$  linked with $a$ with label $2k$ for $k>1$ (here we use the convention that $w^a=a^{-1}wa$).

As  $A_{\st_\Gamma(a)}$ is parabolic, we have an explicit presentation for $A_{\st_\Gamma(a)}$. The generators are the vertices of $\st_\Gamma(a)$ and there are two types of relators that involve $a$: some of the form $[v,a]$ and some of the form $(wa)^k=(aw)^k$ for $k>1$. 

We can modify this presentation using Tietze transformations as follows. We add generators of the form $w_i=w^{a^i}$ for $1\leq i\leq k-2$ and the elements $z_{a,w}$ as above. Writing $w_0$ instead of $w$, we get new relators
\begin{equation}\label{eq:conjw}w_i^a=w_{i+1}\end{equation}
for $i\in \{0,\ldots,k-3\}$,
$$w_{k-2}^{a}=z_{a,w}w_0^{-1}\ldots w_{k-2}^{-1},$$
and furthermore, the relator  $(wa)^k=(aw)^k$ yields
$$az_{a,w}=a^{-k}(wa)^k=a^{-k}(aw)^k=z_{a,w}a.$$

Therefore, we get a new presentation as follows. As generators, we have $a$ together with the set $\Omega_1$ of vertices $v$ in $\lk_\Gamma(a)$ such that the label of the edge $\{a,v\}$ is 2 and with the set $\Omega_2$ consisting of the $k$ elements 
$$w_0,w_1,\ldots,w_{k-2}, z_{a,w}$$
 for each vertex $w$ in $\lk_\Gamma(a)$ such that the label of the edge $\{a,w\}$ is $2k$.

 The relators that involve $a$ are of the following form. For each $v\in\Omega_1$,  $v^a=v$, and for each $w$ linked to $a$ with label $2k$, the relators in equation (\ref{eq:conjw}) and the relator $z_{a,w}^a=z_{a,w}$. 

To understand which are the relators that do not involve $a$, recall that the FC-condition implies that if two vertices $w$ and $u$ are linked to $a$ with labels greater than 2, then they cannot be adjacent. If such a $w$ is linked to some $v\in V$ with $[v,a]=1$, then the label must be 2, so we also get that $[w_i,v]=1$ for $0\leq i\leq k-2$ and $[z_{a,w},v]=1$. 
Moreover, we have the relators given by the subgraph of $\Gamma$ with $\Omega_1$ as vertex set.

Now, let $M$ be the abstract group defined by a presentation having as generators the elements in $\Omega_1$ and $\Omega_2$ together with all the previous relators that do not involve $a$. We may define an action of $a$ by conjugation on this group using the relators that do involve $a$. To check that this action is well-defined we argue as follows. The only instances in which it may happen that two generators $x, y$ of $M$ are related (in $M$) are when either both lie in $\Omega_1$, and in that case the $a$-action leaves both invariant, or if one, say $x$, lies in $\Omega_1$ and the other, say $y$, is one of the $w_i$'s or $z_{a,w}$. In this last case the FC-condition mentioned above implies that the relation between $x$ and $y$ must be of type $[x,y]=1$ and the definition of the $a$-action implies that in all the possible cases we have that $[x^a,y^a]=1$.

Therefore, we can define the (abstract) semidirect product $\langle a\rangle \ltimes M$. But the standard semidirect product presentation for this group is precisely the presentation for $A_{\st_\Gamma(a)}$ that we had before, so both groups are isomorphic.
In fact, looking at the presentation given for $M$, we see that it is the Artin group associated to the graph $\hat{\Gamma}$ with vertices precisely the elements in $\Omega_1\cup\Omega_2$ and edges determining the relators of $M$. 
 Let $L\leq M$ be the subgroup generated by the subgraph determined by $\Omega_1$ and all the vertices of the form $z_{a,w}$ of the $\Omega_2$ set. Then, $L$  together with $a$ generate $C_G(a)$ (see equation (\ref{eq:defz}) and the paragraph before). Moreover, the previous construction implies that $L$ is the Artin group associated to a subgraph of $\hat{\Gamma}$ which is isomorphic to $\lk_\Gamma(a)$ (but with different vertex labels).
 
In order to show the last part of the statement, we assume that $a,b\in V$ are not linked and we set $S$ to be the induced subgraph with vertices $a$ and $b$. Then $A_S$ has no spherical direct factor and the only way that  $S\cup\{t\}$ has one (with $t\in V$) is when $t$ commutes with both $a$ and $b$. Hence, using Godelle's result again we deduce that $C_G(S)\leq A_T$. To conclude, recall that by \cite[Lemma 2.13]{AntolinFoniqi2}, we have that
$$C_{G}(A_S)=C_G(a)\cap C_G(b).$$
\end{proof}

\begin{proposition}\label{EvenFC_normalizer}
Let $G=A_\Gamma$ be an even Artin group of FC-type and $\varnothing\neq S\subseteq\Gamma$.  Let $N_G(A_S)$ be the normalizer of $A_S$ in $G$. Then either $N_G(A_S)=A_S$ or $N_G(A_S)$ embeds in a direct product of Artin subgroups of $G$. More explicitly, either
\begin{itemize}
\item[(i)] $N_G(A_S)=A_S$; or
\item[(ii)] $N_G(A_S)\leq A_{T_1}\times A_{T_2}$ for $\varnothing \neq T_1,T_2\subseteq\Gamma$ disjoint; or
\item[(iii)] $A_S$ is free abelian and $N_G(A_S)\leq L\times A_S$ where $L\leq G$ is a (non-parabolic) Artin group based on a proper subgraph of $\Gamma$.
\end{itemize}
\end{proposition}

\begin{proof} 

We have that $C_G(a)\leq A_{\st_\Gamma(a)}$ (see Lemma \ref{tec_Godelle}), so
$$\bigcap_{a\in V(S)}C_G(a)\leq\bigcap_{a\in V(S)}A_{\st_\Gamma(a)}\leq A_{S\cup\lk_\Gamma(S)},$$
and, therefore,  from (\ref{eq:AF}) we get that
$$N_{G}(A_S)\leq  A_{S\cup\lk_\Gamma(S)}.$$
We may then assume that $\lk_{\Gamma}(S)$ is non-empty.

Consider first the case when $S=\{a\}$ has only one element. Then, as $N_G(A_S)=C_G(a)$ (see (\ref{eq:AF})), using Lemma \ref{tec_Godelle} we have (iii).
% as $C_G(a)\leq A_{\st_\Gamma(a)}$, we might suppose that $\Gamma=\st_\Gamma(a)$, so we have (iii).

Now,  assume that $A_S$ is free abelian. Then, for any $a\in V(S)$, 
$$N_G(A_S)=A_SC_G(A_S)=C_G(A_S)\leq C_G(a),$$
where the first equality is  (\ref{eq:AF}).
Using the previous case we get again (iii).

In conclusion, from now on we deal with the case when $A_S$ is not free abelian. Suppose that there is some vertex $v$ in $\lk_\Gamma(S)$ and $a\in V(S)$ such that the edge in $\Gamma$ between $v$ and $a$ has label $2k$ for $k>1$.

Suppose that there is some $a\neq b\in V(S)$ with $[a,b]\neq 1$. Note that $a$ and $b$ cannot be linked in $\Gamma$ because, if they were, as they do not commute, the triangle $a,b,v$ would have at least two labels greater than 2, which is forbidden for even Artin groups of FC-type. We claim that, in this case, the vertex $v$ can be removed from $\Gamma$ without changing the normalizer, i.e., that if we set $G_1$ to be $A_{\Gamma\setminus\{v\}}$, then \[N_G(A_S)=N_{G_1}(A_S).\]
For this, it is enough to show that  $C_G(A_S)=C_{G_1}(A_S)$ (recall (\ref{eq:AF})). 
Consider the subgraph induced by the three vertices $a$, $b$, $v$. Lemma \ref{tec_Godelle} implies  that $C_G(a)\cap C_G(b)=C_{G_1}(a)\cap C_{G_1}(b)$. Therefore, 
\[
C_G(A_S)=\bigcap_{u\in V(S)}C_G(u)\cap C_G(a)\cap C_G(b)=\bigcap_{u\in V(S)}C_G(u)\cap C_{G_1}(a)\cap C_{G_1}(b)=\]\[
=\bigcap_{u\in V(S)}C_G(u)\cap G_1=\bigcap_{u\in V(S)}C_{G_1}(u)=C_{G_1}(A_S).
\]

This implies that we can successively remove from $\Gamma$ vertices of $\Gamma\setminus S$ without affecting the normalizer until we get a subgraph $S\subseteq\Gamma_2\subseteq\Gamma$ such that for any vertex $w$ of $\lk_{\Gamma_2}(S)$ and any $a\in V(S)$ so that $w$ and $a$ are linked by an edge of label $2k$ with $k>1$, the vertex $a$ commutes with any other vertex in $S$, that is, it is central in $A_S$. Hence, there is some $Z\subseteq S$ (possibly empty) with $A_Z$ central in $A_S$ and such that all the edges between $S$ and $\lk_{\Gamma_2}(S)$ with label greater than 2 have an endpoint in $Z$. As $S$ is not free abelian, $\varnothing \neq S\setminus Z$, and the elements of $A_{S\setminus Z}$ commute with the elements of  $A_{\lk_{\Gamma_2}(S)}$. Therefore, for $G_2=A_{\Gamma_2}$ we have that
$$N_G(A_S)=N_{G_2}(A_S)\leq A_{S\setminus Z\cup Z\cup\lk_{\Gamma_2}(S)}=A_{S\setminus Z}\times A_{Z\cup\lk_{\Gamma_2}(S)},$$
so either we get (i) or (ii).

\end{proof}

\begin{theorem}\label{hyperbolic}
    Let $G=A_\Gamma$ be an even Artin group of FC-type and $H\leq A_\Gamma$ a subgroup. Suppose that there is a vertex $w$ in $\Gamma$ such that $H$ is not a subgroup of $A_{\st_\Gamma(w)}^g$ or of $A_{\Gamma\setminus\{w\}}^g$ for any $g\in G$. Consider the decomposition of $A_\Gamma$ of the form
    \[ A_\Gamma=A_{\Gamma\setminus\{w\}}\ast_{A_{\lk_\Gamma(w)}}A_{\st_\Gamma(w)},\]
and let $T$ be the associated Bass-Serre tree. Then there is an element $h\in H$ acting hyperbolically on $T$.
       \end{theorem}
\begin{proof}
We follow the steps taken by Minasyan-Osin in \cite[Lemma 6.10]{MinasyanOsin} where they show the analogous result for graph products of groups.

If $X\subseteq G$, we define $P_{c_\Gamma}(X)$ to be the unique minimal parabolic subgroup containing $X$. We first show that such a subgroup is well-defined.

\begin{claim}\label{Claim1}
For $X\subseteq G$, there exists a unique minimal parabolic subgroup of $G$ containing $X$.
\end{claim}
\begin{proof}[Proof of Claim \ref{Claim1}]
Since $G$ is parabolic, the collection of parabolic subgroups containing $X$ is non-empty. Taking into account that the intersection of parabolic subgroups if parabolic (see \cite[Theorem 1.1]{AntolinFoniqi}), we can take the intersection of parabolic subgroups of $G$ containing $X$.
%Since $\Gamma$ is a finite graph, there is $S\subseteq V$ a subset of minimal cardinality such that $X \subseteq g_1 A_S g_1^{-1}$ for some $g_1 \in G$. Let us show that $g_1 A_S g_1^{-1}$ is a minimal parabolic subgroup containing $X$.

%Suppose that there is some parabolic subgroup $g_2 A_T g_2^{-1}$ of $G$ such that $X \subseteq g_2 A_T g_2^{-1} \subseteq g_1 A_S g_1^{-1}$. According to \cite[Lemma 3.2]{AntolinFoniqi}, $T \subseteq S$, so $A_T= A_{T\cap S}$. Minimality of $|S|$ implies that $|S| \leq |S \cap T|$, and therefore $S= T \cap S$ and $A_S= A_{T \cap S}= A_T$. Thus, applying \cite[Lemma 3.7]{AntolinFoniqi} we can conclude that $g_2 A_T g_{2}^{-1}= g_2 A_S g_2^{-1}= g_1 A_S g_1^{-1}$, as required.

%The uniqueness of this parabolic subgroup follows just from the fact that the intersection of parabolic subgroups is parabolic, see \cite[Theorem 1.1]{AntolinFoniqi}.
\end{proof}

The next step is to prove a nice property of these parabolic subgroups.

\begin{claim}\label{Claim2}
Let $X \subseteq G$. Then there is a finite subset $X' \subseteq X$ such that $P_{c_\Gamma}(X)= P_{c_\Gamma}(X')$.   
\end{claim}
\begin{proof}[Proof of Claim \ref{Claim2}]
By Claim \ref{Claim1}, for every finite subset $Y \subseteq X$, there exist $S(Y) \subseteq V$ and $f(Y)\in G$ such that
\[ P_{c_\Gamma}(Y)= f(Y) A_{S(Y)} f(Y)^{-1},\]
where  we abuse notation and we denote by $S(Y)$ the subgraph of $\Gamma$ induced by $S(Y)$. Since $\Gamma$ is finite, the function, that assigns to each finite subset $Y\subseteq X$ the integer value $|S(Y)|$, attains its maximum on some finite subset $X' \subseteq X$. Hence, to prove the result it is enough to show that $X\subseteq P_{c_\Gamma}(X')$.

For any $x\in X$, we can find $T\subseteq V$ and $g\in G$ such that $P_{c_\Gamma}(X' \cup \{x\})= g A_T g^{-1}$. By the choice of $X'$, $|T| \leq |S|$ and $f A_S f^{-1} \leq g A_T g^{-1}$, where $S \coloneqq S(X')$ and $f \coloneqq f(X')$. By \cite[Lemma 3.2]{AntolinFoniqi}, $S\subseteq T$, which implies that $S= T$. Finally, the inclusion $f A_T f^{-1} \subseteq g A_T g^{-1}$ together with \cite[Lemma 3.7]{AntolinFoniqi} yield \[f A_S f^{-1}= f A_T f^{-1}= g A_T g^{-1},\] allowing to conclude that $x\in P_{c_\Gamma}(X' \cup \{x\})= P_{c_\Gamma}(X')$ for all $x\in X$.
\end{proof}
Coming back to the main result, suppose by contradiction that $H$ contains no hyperbolic elements, so every element of $H$ is elliptic.

By Claim \ref{Claim2}, there is a finite subset $X'\subseteq H$ such that
 \[ P_{c_\Gamma}(X')= P_{c_\Gamma}(H).\]
The subgroup $H' \coloneqq \langle X' \rangle \leq H$ is finitely generated and every element of $H'$ fixes a vertex of $T$, so $H'$ fixes some vertex of $T$ (see, for instance, \cite[(3.4)]{Tits2}). Hence, there is $f\in G$ such that $H' \leq f A_D f^{-1}$, where $D$ is either $\st_{\Gamma}(w)$ or $\Gamma \setminus \{w\}$. It follows that
\[ H \leq P_{c_\Gamma}(H)= P_{c_\Gamma}(X')= P_{c_\Gamma}(H') \leq f A_D f^{-1},\]
which contradicts our assumption.
\end{proof}

We end up the section by proving the existence of WPD elements in certain subgroups of even Artin groups of FC-type and, as a consequence, showing that these subgroups are acylindrically hyperbolic.

\begin{theorem}\label{WPD}
Let $A_\Gamma$ be an even Artin group of FC-type and $H\leq A_\Gamma$ a subgroup which is not virtually cyclic. Suppose that there is a vertex $w$ in $\Gamma$ such that $H$ is not a subgroup of $A_{\st_\Gamma(w)}^g$ or of $A_{\Gamma\setminus\{w\}}^g$ for any $g\in G$. Consider the decomposition of $A_\Gamma$ of the form
    \[ A_\Gamma=A_{\Gamma\setminus\{w\}}\ast_{A_{\lk_\Gamma(w)}}A_{\st_\Gamma(w)}.\]
    Then either
\begin{itemize}
\item[(i)] there are Artin groups $A_{T_1}$, $A_{T_2}$ with $0<|V(T_1)|,|V(T_2)|<|V(\Gamma)|$ such that $H\leq A_{T_1}\times A_{T_2}$; or
\item[(ii)] $H$ is acylindrically hyperbolic.
\end{itemize}
\end{theorem}

\begin{proof}
The proof is similar to the one of \cite[Corollary 6.19]{MinasyanOsin}. Hence, we refer the reader to such paper for a deeper understanding of the main concepts and results.

Let $T$ be the  Bass-Serre tree associated to the decomposition in the statement. We claim that under the statement conditions, there is a hyperbolic element $g\in H$ such that if $\mathcal{L}$ is the axis of $g$ in $T$ and $\text{PStab}_{A_\Gamma}(\mathcal{L})$ is the pointwise stabilizer of $\mathcal{L}$, then $\text{PStab}_{A_\Gamma}(\mathcal{L})$ is a parabolic  subgroup of $A_\Gamma$ normalized by $H$.

Let us first show that for any hyperbolic isometry $g$, if $\mathcal{L}$ is the axis of $g$ in $T$, then $\text{PStab}_{A_\Gamma}(\mathcal{L})$ is parabolic. Indeed, by definition, $\text{PStab}_{A_\Gamma}(\mathcal{L})$ equals
\[ \text{PStab}_{A_\Gamma}(\mathcal{L})= \bigcap_{e\in \mathcal{E}} \text{PStab}_{A_\Gamma}(e),\]
where $\mathcal{E}$ is the set of all edges of $\mathcal{L}$. Since each $\text{PStab}_{A_\Gamma}(e)$ is a conjugate of $A_{\lk_\Gamma(w)}$, by \cite[Corollary 5.3]{AntolinFoniqi} we get that $\text{PStab}_{A_\Gamma}(\mathcal{L})$ is a parabolic subgroup of $A_\Gamma$. In fact, the same corollary shows that
\[ \text{PStab}_{A_\Gamma}(\mathcal{L})= \bigcap_{e\in \mathcal{E}^\prime} \text{PStab}_{A_\Gamma}(e)= \text{PStab}_{A_\Gamma}(\mathcal{E}^\prime),\]
for some finite collection $\mathcal{E}^\prime$ of edges of $\mathcal{L}$. As any finite collection of edges of $\mathcal{L}$ is contained in a finite subsegment of $T$, we can find two vertices $x$ and $y$ in $\mathcal{L}$ such that
\[ \text{PStab}_{A_\Gamma}(\mathcal{L})= \text{PStab}_{A_\Gamma}([x,y])= \text{PStab}_{A_\Gamma}(\{x,y\}).\]
As a summary, we have just proved that if $g$ is a hyperbolic element and $\mathcal{L}$ is its axis, then $\text{PStab}_{A_\Gamma}(\mathcal{L})$ is parabolic, and moreover, that there are two vertices $x$ and $y$ such that $\text{PStab}_{A_\Gamma}(\mathcal{L})=\text{PStab}_{A_\Gamma}(\{x,y\}).$

Let us secondly prove that we can find a hyperbolic element $g\in H$ such that $\text{PStab}_{A_\Gamma}(\mathcal{L})$ is normalized by $H$ in $A_\Gamma$. Following \cite[Definition 6.2]{MinasyanOsin}, if $P= f A_S f^{-1}$ is a parabolic subgroup of $A_\Gamma$ for some subgraph $S$ of $\Gamma$ and $f\in A_\Gamma$, then the \emph{parabolic dimension} $\text{pdim}_\Gamma(P)$ of $P$ is the number of vertices of $S$. Then \cite[Lemma 3.2]{AntolinFoniqi} ensures that $\text{pdim}_\Gamma(P)$ is well-defined, as the result shows that if $g_1 A_ S g_1^{-1} \subseteq g_2 A_T g_2^{-1}$ for some subgraphs $S, T$ of $\Gamma$ and $g_1, g_2 \in A_\Gamma$, then $S \subseteq T$.

We claim that if $g$ is a hyperbolic element such that $\text{pdim}_\Gamma(\text{PStab}_{A_\Gamma}(\mathcal{L}))$ is minimal, then $\text{PStab}_{A_\Gamma}(\mathcal{L})$ is normalized by $H$ in $A_\Gamma$. This is exactly shown as \cite[Proposition 6.13]{MinasyanOsin}. The properties used here are proved in \cite{AntolinFoniqi}, and they correspond to \cite[Lemma 3.1, Lemma 3.2, Lemma 3.4]{AntolinFoniqi}.

By Theorem \ref{hyperbolic} there are indeed hyperbolic elements in $H$ for the given action on the Bass-Serre tree $T$. If we consider $h$ an element in $H$ which is hyperbolic and such that $\text{pdim}_\Gamma(\text{PStab}_{A_\Gamma}(\mathcal{L}))$ is minimal, where $\mathcal{L}$ is the axis of $h$, then we have shown that $\text{PStab}_{A_\Gamma}(\mathcal{L})$ is normalized by $H$ and parabolic in $A_\Gamma$, so there is a subgraph $S$ of $\Gamma$ and $f\in A_\Gamma$ such that $fA_S f^{-1}$ is normalized by $H$. But then $H\leq N_G(f A_S f^{-1})$ and as $H$ is not a subgroup of $f A_S f^{-1}$ (otherwise $H$ would fix a point in $T$) by Proposition \ref{EvenFC_normalizer}, we get (i). Otherwise, we conclude that $\text{PStab}_{A_\Gamma}(\mathcal{L})$ is trivial. We have shown that there are two vertices $x$ and $y$ such that $\text{PStab}_{A_\Gamma}(\mathcal{L})=\text{PStab}_{A_\Gamma}(\{x,y\})$, so \cite[Corollary 4.3]{MinasyanOsin} implies that $h$ satisfies the WPD condition. Then, $A_{\lk_\Gamma(w)}$ is weakly malnormal, so from \cite[Corollary 2.2]{MinasyanOsin}, it follows that the action of $H$ on the associated Bass-Serre tree is acylindrically hyperbolic (since we are assuming that $H$ is not virtually cyclic).
\end{proof}

\begin{corollary}\label{acylindrically hyperbolic}
 Let $A_\Gamma$ be an even Artin group of FC-type and $H\leq A_\Gamma$ a subgroup which is not virtually cyclic. Then either
\begin{itemize}
\item[(i)] there are $1\neq N_1,N_2\normal H$ with $N_1\cap N_2=1$ (so $N_1\times N_2\leq H$); or
\item[(ii)] $H$ is acylindrically hyperbolic.    
\end{itemize}
\end{corollary}

\begin{proof}

Let us argue by induction on $|V(\Gamma)|$.  Assume first that $H$ is a subgroup of $A_{T_1}\times A_{T_2}$ for $T_1,T_2$ graphs such that $|V(T_1)|,|V(T_2)|$ are smaller than $|V(\Gamma)|$. Then if $H\cap A_{T_1}=1$, we see that $H$ is isomorphic to a subgroup of $A_{T_2}$ and we may use the inductive hypothesis. The same happens if $H\cap A_{T_2}=1$. Hence, we may assume that $H\cap A_{T_1}, H\cap A_{T_2}$ are non-trivial. Denoting $N_j$ to be $H\cap A_{T_j}$ for $j\in \{1,2\}$, we obtain (i) in this case.

If $H$ is a subgroup of $A_\Delta$ such that $\Delta$ is a complete graph, then either we can embed $H$ in a decomposition as before or $A_\Delta$ is dihedral (recall that $\Delta$ is even of FC-type). If $G=A_\Delta$ is dihedral associated to the edge $\{a,b\}$ with label $2k$, the defining relator is $(ab)^k=(ba)^k$. Setting $c$ to be $ab$, this relator can be transformed into $c^kb=bc^k$. Then, if we denote by $C$ the subgroup generated by $c^k$, we see that $C$ is central in $G$ (in fact, it is the center of $G$). Let $F$ be the normal subgroup of $G$ normally generated by $b$. It is known (see, for example, \cite[Lemma 2.5]{AntolinFoniqi2}) that $F$ is free and $C\times F$ has finite index in $G$. Moreover, using the presentation of $G$ in terms of $c$ and $b$ we have that $G \slash C=\Z\ast \Z_k$. Now, if $H\cap C=1$, then $H$ is isomorphic to a subgroup of $G\slash C$ and by \cite[Theorem 2.12]{MinasyanOsin} we conclude that $H$ is acylindrically hyperbolic. Thus, we assume that $H\cap C\neq 1$. But as $H$ is not virtually cyclic, $H\cap F\neq 1$ and we are in case (i).

Therefore, we may assume that the graph $\Gamma$ is not complete, so that there is some $w$ of $\Gamma$ such that $\st_\Gamma(w)$ is a proper subgraph. Then, if $H$ is a subgroup of 
 $A_{\st_\Gamma(w)}^g$ or of $A_{\Gamma\setminus\{w\}}^g$ for some $g\in A_\Gamma$, $H^{g^{-1}}$ is a subgroup of the corresponding (proper) standard parabolic subgroup and we may argue by induction. Otherwise, we may apply Theorem \ref{WPD} together with the first paragraph to conclude the proof.
\end{proof}

Recall that $N_1$ and $N_2$ above must be infinite because even Artin groups of FC-type are torsion-free. Acylindrically hyperbolic groups cannot contain infinite normal subgroups as in (1) of Corollary \ref{acylindrically hyperbolic} (see \cite[Theorem 3.7 (b), (c)]{MinasyanOsin}), so this result implies that for subgroups of even Artin groups of FC-type this obstruction characterizes acylindrical hyperbolicity.

%\section{Conflict of interest}
%On behalf of all authors, the corresponding author states that there is no conflict of interest. 

%\section{Data availability} 
%We do not analyze or generate any datasets, because our work proceeds within a theoretical and mathematical approach. One can obtain the relevant materials from the references below.

\end{document}